\documentclass[12pt]{huber_article}

\usepackage[hmargin=1.5in,vmargin=1.5in]{geometry}
\usepackage{graphicx}
\usepackage{verbatim}
\usepackage{enumerate}
\usepackage{fullpage}

\newcommand{\real}{\mathbb{R}}
\newcommand{\prob}{\mathbb{P}}
\newcommand{\mean}{\mathbb{E}}

\newcommand{\ind}{{\bf 1}}

\usepackage[section]{algorithm}
\usepackage{algorithmic}

\usepackage{amsmath,amsfonts,amsthm}

\theoremstyle{plain}
\newtheorem{theorem}{Theorem}

\newtheorem{lemma}{Lemma}

\theoremstyle{definition}

\begin{document}

\title{Bounds on the artificial phase transition for
perfect simulation of repulsive point processes}

\maketitle

{\bfseries \sffamily Mark L. Huber} \par
{\slshape Department of Mathematics and Computer Science, 
  Claremont McKenna College} \par
{\slshape mhuber@cmc.edu}
\vskip 1em
{\bfseries \sffamily Elise McCall} \par
{\slshape 
  Massachusetts Institute of Technology} \par
{\slshape elise@mit.edu}
\vskip 1em
{\bfseries \sffamily Daniel Rozenfeld} \par
{\slshape 
  Harvey Mudd College} \par
{\slshape Daniel\_J\_Rozenfeld@hmc.edu}
\vskip 1em
{\bfseries \sffamily Jason Xu} \par
{\slshape University of Arizona} \par
{\slshape qxu@email.arizona.edu}

{\abstract Repulsive point processes arise in models where competition
forces entities to be more spread apart than if placed independently.
Simulation of these types of processes can be accomplished using
dominated coupling from the past with a running time that varies
as the intensity of the number of points.  These algorithms 
usually exhibit what is called an artificial phase transition, where below
a critical intensity the algorithm runs in finite expected time, but above
the critical intensity the expected number of steps is infinite. 
Here the artificial phase transition is 
examined.  In particular, an earlier lower bound on this artificial 
phase transition is improved by including a new type of term in the 
analysis.  In addition, the results of computer 
experiments to locate the transition are presented.}

\section{Introduction}

A spatial point process is a collection of points in a set $S$.  In
most applications, $S$ is a continuous space and all of the points
are distinct.  For instance, the locations of trees in a 
forest~\cite{mollerw2007} and the locations of cities in a 
country~\cite{glasst1971} can be modeled using spatial point
processes.

One simple spatial point process is the  Poisson point process.  Suppose that
$S$ is a separable set in ${\mathbb R}^d$ with bounded Lebesgue measure.
The basic Poisson point process is the outcome of the following 
algorithm.  First choose a random number of points $N$ according to a 
Poisson distribution with parameter $\lambda \mu(S)$ (so
$\prob(N = i) = \exp(-\lambda \mu) (\lambda \mu)^i / i!$ for nonnegative
integers $i$.)  Here
$\mu$ is Lebesgue measure and $\lambda \in \real$ is a parameter of the 
model. 
Next, choose points $X_1,\ldots,X_n$ independently and uniformly from 
the set $S$.  The resulting set $\{X_1,\ldots,X_N\}$ is a Poisson
point process.

Since the points are drawn independently, this model fails to
capture situations where the locations of points are
not independent.  In both the forest and cities examples mentioned
earlier, the points tend to be farther apart than in the independent
situation since the entities involved
 are competing for space and resources.  The points
appear to act as particles with the same charge, and so they
exhibit repulsion.

There are several ways to account for this repulsion.  One type of
model is called a pairwise interaction point process.  An example of
such a model is the hard core gas model,
where each point is surrounded by a {\em core} of radius
$R / 2$ which is hard in the sense that two cores are not allowed to overlap.
In other words, all of the points of the process must be at least distance
$R$ away from each other, where $R$ is a parameter of the model.

In frequentist approaches, this model can be used to construct maximum
likelihood estimators for $R$ and $\lambda$.  In Bayesian approaches,
this model (together with a prior on $\lambda$ and $R$) can be used to
build a posterior for the parameters.  In either instance, evaluation of
needed quantities is usually accomplished through simulation:  drawing
samples from the model.

Kendall and M{\o}ller~\cite{kendallm2000} 
showed how to draw samples from the hard core gas model by using
dominated coupling from the past (dcftp.)  A previous analysis 
had shown that when using the standard Euclidean distance, 
this method was provably fast when 
$\lambda < 1 / [\pi R^2]$~\cite{huber2009f}.  
In this work we build upon this analysis, providing a wider set of 
conditions on $\lambda$ and $R$ for the dcftp method to run quickly.  The
original argument used a term depending on the number of points in the 
configuration, while the new method uses the number of points as well as the
area spanned by these points.  This extra area term is what leads to the 
stronger proof.  For ease of exposition we use the Euclidean metric to measure
the distance between points and only operate in $\real^2$ throughout
this work; we simply note that the same argument 
can easily be applied to any metric and to problems in higher dimensions.

The remainder of the work is organized as follows.  Section~\ref{SEC:setup}
describes how to build the hard core gas model in detail.  The following
section illustrates how to construct a continuous time Markov chain 
whose
stationary distribution matches the model.  Section~\ref{SEC:dcftp} then
explains how dominated coupling from the past can use the Markov chain
to obtain draws exactly from the target distribution.  Next, 
Section~\ref{SEC:bounding} gives our new result:  
improved sufficient conditions on the parameters
of the model for dominated coupling from the past to operate quickly.
Section~\ref{SEC:experimental} gives computer results to complement the 
theoretical
results of the previous section, and we close with our conclusions.

\section{Setting up the hard core gas model}
\label{SEC:setup}

There are several methods for describing the hard core gas model; 
in this section it is constructed 
as the outcome of the following algorithm.  
Begin with a parameter $\lambda$ and
a space $S$ with $\mu(S) < \infty$, where $\mu$ is Lebesgue measure.
As in the previous section, start by choosing the number of points $N$
via a Poisson distribution with parameter $\lambda \cdot \mu(S)$, and
then choose $\{X_1,\ldots,X_N\}$ independently and uniformly from $S$.
%Recall that $X \sim \pois(\mu)$ means 
%$\prob(X = i) = \exp(-\mu) \mu^i / i!$ where
%$i \in \{0,1,2,\ldots\}$.  
(It is easy to generalize this setup to
more general measures, but for most applications the measure of interest
 is Lebesgue or 
absolutely continuous with respect to Lebesgue measure.)

The outcome of the above procedure is a Poisson point process with parameter
$\lambda$.  Now fix $R \in (0,\infty)$.  Run
the following procedure.  Draw a Poisson point process $X$.
If any two points are within distance
$R$ of each other, throw away the entire point process and start over
by drawing a new Poisson point process.  Continue drawing point
processes until a configuration is found where every point is at least
distance $R$ from its closest neighbor.

Because the chance of acceptance decreases as the number of pairs within
range increases, a draw from a hard core process will have fewer points
that are farther apart than in an unmodified Poisson point process
with parameter $\lambda$.

\section{Continuous time birth death chains}

While the method described in the previous section will always terminate
with probability 1, if $\lambda$ and $S$ are very large
the probability of rejecting the configuration and starting over will be 
prohibitively close to 1.  This makes the method unusable in practice.

To avoid this problem, Markov chains are often used instead.  
A Markov chain is a stochastic
process where the future distribution of the state depends only upon
the current state, and not upon the past history of the chain.  The classic
example is shuffling a deck of cards, where the chance of doing a 
particular shuffle move does not depend on the past history of the deck.
Under mild conditions, the distribution of the configuration will
approach a stationary distribution.  Again using the example of the cards,
under most shuffling schemes, the cards quickly approach the distribution
that is uniform over the set of permutations of the cards.

For point processes, a particular type of Markov chain introduced by 
Preston~\cite{preston1977} is called a spatial birth-death process.
In this chain, moves either add a point (called
a {\em birth}) or remove a point (called a {\em death}).
To make moves in the chain,
think of a sequence of alarm
clocks.  The space itself has a birth clock where the time until the alarm
goes off is a random variable that has an exponential distribution
with mean 
$1/[\lambda \cdot \mu(S)]$.  When this alarm clock goes
off a point is born and added to the configuration at a uniformly
chosen location in $S$.  

When a point is born, it is given a death alarm clock that is an exponential
random variable with mean 1.  When a death alarm clock ``goes off'', the
point in question is removed from the configuration.

Now in order to create a birth death process whose stationary distribution
is the hard core gas model, 
it is necessary to sometimes reject a birth.  That is,
even though the birth clock has ``gone off'', a point will only be added to
the configuration with a certain probability that depends on the locations of
the proposed birth point and the points in the rest of the configuration.

For the hard core gas model, this works as follows.  Suppose
a point $v$ is proposed to be born to state $x$.  Then accept the birth
only if there are no points within distance $R$ of $v$ in the current
configuration $x$.  Otherwise, reject the birth, and do not add $v$.
When a point $v$ is not born because a point $w$ in $x$ is within
distance $R$ of $v$, say that point $w$ {\em blocks} $v$.

\section{Dominated coupling from the past}
\label{SEC:dcftp}

The natural question with shuffling is:  how many moves are needed before the 
cards are close to being uniformly permuted?  For birth death chains,
the question is similar:  how many births and deaths are needed before
the state is close to the distribution described by the hard core gas model?
Fortunately, it turns out that 
it is not necessary to determine the mixing time of a 
Markov chain in order to draw samples from the stationary distribution!

Dominated coupling from the past (dcftp) was created by Kendall and
M{\o}ller~\cite{kendallm2000} to draw samples exactly from the hard core
gas 
model without having to use the acceptance/rejection procedure.  It works
in conjunction with a continuous time Markov chain where points are ``born''
and added to the configuration, and ``die'' and are removed from the system.

The time necessary to run dcftp is related to the {\em clan of descendants}
(cod) of a point, defined as follows.  The zeroth generation of 
the cod is the point itself.  The first generation consists of those 
proposed points that
are born within distance $R$ of the zeroth generation while that initial
point is still alive.  If the initial point dies before any proposed point is
born within distance $R$, then the entire cod consists solely of the 
initial point.

Suppose that proposed 
points are born within distance R of the initial point
before that initial point dies.  
The remaining generations are defined recursively in a similar fashion.
A point joins the cod at the generation $k$ if a) it is proposed to be 
born within distance
$R$ of a generation $k - 1$ point that is still alive and b) 
$k$ is the minimum value for which a) holds.

Then the cod is the union of these points over all generations.
Roughly speaking, the cod is the set of points whose presence (or lack of
presence) in the configuration can be traced back to the original ancestor
point.  If the cod is small, it means that the influence of a particular 
point is not felt forever in the configuration, but rapidly dissipates.
The running time of dcftp is proportional to the size of the cod.  If there
is a chance that the cod grows indefinitely, dcftp has the same chance of
taking forever to
generate a sample, so the algorithm is only useful when the cod is finite
with probability 1.  These ideas are made precise in~\cite{kendallm2000}.

\section{Bounding the size of the clan of descendants}
\label{SEC:bounding}

In order to bound the size of the clan of descendants, let $C_t$ denote
the points in the cod (of any generation) 
at time $t$.  Initially, $C_0 = \{v\}$, the single
ancestor point.  There are two possibilities when the cod changes: 
either the size of $C_t$
(denoted $\# C_t$) increases by one or it decreases by one.  
If a point $w$ is proposed to be born within
distance $R$ of $v$ then the point
$w$ is added to the clan of descendants, and $\#C_t$
increases by 1.  On the other hand, when $v$ dies,
it is removed from $C_t$ and $\# C_t$ decreases by 1.

We wish to
show that $\# C_t$ converges to 0 (so that $C_t = \emptyset$)
with probability 1 after a finite number of
births and deaths that affect the cod.  In particular,

\begin{theorem}
For $\lambda < [8/(3\sqrt{3} + 4\pi) ] / R^2$, 
the expected number of 
births and deaths that affect the cod is bounded above by
\[
\left[ \frac{8/(3\sqrt 3 + 4\pi)}{R^2} - \lambda \right]^{-1}.
\]
\end{theorem}

As noted in the introduction, a similar previous result in~\cite{huber2009f}
had a constant of $1 / \pi \approx .3183$ in front of the 
$R^{-2}$ factor, whereas this
new result has $8/(3 \sqrt{3} + 4\pi) \approx .4503.$  Hence this result
proves the efficacy of the dcftp method (and mixing time of the chain)
over values of $\lambda$ that are 41\% larger than previously known.

Before proving this theorem, we first develop some notation and facts
that will be useful.
As earlier, let $C_t$ denote the set of points in the cod.  We are only 
interested in how $C_t$ changes with births and deaths.  Hence let 
$t_i$ denote the time of the $i$th event that is either 
a death of a point
in the cod, or the proposed birth of a point within distance 
$R$ of the cod.  Let $D_i = C_{t_i}$, so $D_i$ represents the cod after
$i$ such events have occurred.  Let $\# D_i$ denote the number of points
in this set.

For a configuration $x$, let $A(x)$ denote the Lebesgue measure of the
region within distance $R$ of at least one point in $x$.  In particular,
$A(D_i)$ is the measure of the area of the region within distance $R$
of points in the cod.  So $A(D_i)$ 
is proportional to the rate at which births occur
that increase $\# D_i$ by 1.  Our first lemma limits the average 
area that is added when such a birth occurs.

\begin{lemma}
$\mean[A(D_{i+1}) - A(D_i)|\operatorname{a\ birth\ is\ accepted}] 
  < R^2 3\sqrt{3} / 4.$
\end{lemma}

\begin{proof}
Let $w$ be a proposed birth point.  Then in order to add to the clan of
descendants, $w$ must be within distance $R$ of a point $v$ of $D_i$.
The area of the new setup does not increase by $\pi R^2$, however, since only
the region within $R$ of $w$ and not within $R$ of $v$ can be added area.  
Because $w$ 
is conditioned to lie within distance $R$ of $v$, the distance between
centers is a random variable with density 
$f_r(a) = (2a/R^2) \cdot \ind(0 \leq a \leq R)$.
Therefore, the expected area added can be written as:
\[
\mean[\textrm{new area}] \leq
  \int_0^R \frac{2a}{R^2} 
   \left[\pi R^2 - 4 \int_{a/2}^R \sqrt{R^2 - x^2} \ dx \right] \ da
  =  R^2 3 \sqrt{3}/4.
\]
This is an upper bound on the expected new area because $w$ might be
within distance $R$ of other points in $D_i$ as well, which would reduce
the new area added.
\end{proof}

The last lemma gives an upper bound on the area added when a birth occurs.
The next lemma gives a lower bound on the area removed when a death occurs.

\begin{lemma}
$\mean[A(D_{i+1}) - A(D_i)|\operatorname{a \ death \ occurs}]
 \geq [2 A(D_i) / \#D_i] - \pi R^2 .$
\end{lemma}

\begin{proof}
Let $A_k$ denote the area of the region that is within distance $R$ of
exactly $k$ points of $D_i$.  Then
\[
\pi R^2 \# D_i = A_1 + 2 A_2 + 3 A_3 + \cdots + (\# D_i) A_{\#D_i},
\]
and $A(D_i) = A_1 + A_2 + A_3 + \cdots + A_{\#D_i}$.  
Therefore 
\[
2 A(D_i) - \pi R^2 = A_1 - A_3 - 2A_4 - \cdots - (\# D_i - 2)A_{\#D_i} 
\leq A_1.
\]

If the points in $D_i$ are labeled $1,2,\ldots,\#D_i$, then 
$A_1 = a_1 + a_2 + \cdots + a_{\#D_1}$, where $a_k$ is the area of the region
within distance $R$ of point $i$ and no other points.  When a death occurs,
every point in $\#D_i$ is equally likely to be chosen to be removed, so
the average area removed is:
\[
\frac{1}{\#D_i} a_1 + \cdots \frac{1}{\#D_i}a_{\#D_i} =
  \frac{1}{\#D_i} A_1 \leq \frac{2 A(D_i)}{\#D_i} - \pi R^2.
\]
\end{proof}

We are now ready to prove the theorem.

\begin{proof}
For a configuration $x$, let $\phi(x) = A(x) + c \cdot \#x,$ where
$c$ is a constant to be chosen later.  Note that 
$\phi(x)$ is positive unless $x$ is the empty configuration, in which
case it equals 0.  Let $\tau = \inf\{i:D_i = \emptyset\}$.  
Using $a \wedge b$ to denote the minimum of $a$ and $b$, 
we shall show that $\phi(D_{i \wedge \tau}) + (i \wedge \tau) \delta$
is a supermartingale with
$\delta = [2 - \lambda R^2(3\sqrt 3 / 4)] / [1 + \lambda]$.  The
rest of the result 
then follows as a consequence
of the Optional Sampling Theorem (OST).  See Chapter 5 of \cite{durrett2010}
for a description of supermartingales and the OST.

When $i \geq \tau$, $\phi(D_{i \wedge \tau}) + (i \wedge \tau)\delta$ is
identically zero, and so trivially is a supermartingale.

When $i < \tau$, $\phi(D_{i+1})$ either grows when a birth
occurs in the cod, or shrinks when a death occurs.  First consider how
$\#D_i$ changes.  Births occur at rate $\lambda A(D_i)$, and deaths at rate
$\#D_i$.  Hence the probability that an event that changes $\#D_i$ is 
a birth is $A(D_i) / (A(D_i) + \#D_i)$, with the rest of the probability
going towards deaths.  So
\begin{align*}
\mean[\# D_{i+1} - \# D_i| \phi(D_i)] &=
  \mean[ \mean[\# D_{i+1} | D_i] |  \phi(D_i)] \\
 &\leq \mean\left[\ind(i < \tau) \left( \frac{\lambda A(D_i)}{A(D_i) + \#D_i}
       - \frac{\# D_i}{A(D_i) + \#D_i}
    \right) | \phi(D_i) \right].
\end{align*}
(The analysis in~\cite{huber2009f} only considered this term in $\phi$,
which is why the result is weaker than what is given here.)

From our
first lemma, a birth increases (on average) the area covered by the cod
by at most $R^2 3 \sqrt{3} / 4$.  
Our second lemma provides a lower bound on the 
average area removed when a death occurs.  Combining these results yields
\begin{align*}
\mean&[A(D_{i+1}) - A(D_i)|A(D_i)] \\
 &\leq \mean\left[\ind(i < \tau) \left( \frac{\lambda A(D_i)}{A(D_i) + \#D_i}
       R^2 \frac{3\sqrt{3}}{4}
       - \frac{\# D_i}{A(D_i) + \#D_i}\left(\frac{2A(D_i)}{\#D_i} - \pi R^2 
    \right)
    \right) | \phi(D_i) \right] \\
\end{align*}
Note $\ind(i < \tau)$ is measurable with respect to $\phi(D_i)$, and adding
the terms gives:
\[
\mean[\phi(D_{i+1}) - \phi(D_i)| \phi(D_i)]
 \leq \ind(i < \tau) 
  \mean \left[ \frac{A(D_i) (\lambda((R^2 3\sqrt{3}/4) + c) - 2)
    + \#D_i (\pi R^2 - c)}
   {A(D_i) + \# D_i}
  \right].
\]
Now $c$ can be set to
\[
c = \frac{\pi R^2 + 2 - \lambda R^2 (3 \sqrt 3 / 4)}{1 + \lambda},
\]
so that 
\[
\mean[\phi(D_{i+1}) - \phi(D_i)| \phi(D_i)]
 \leq \ind(i < \tau) 
  \mean \left[ \frac{A(D_i) (-\delta)
    + \#D_i (-\delta)}
   {A(D_i) + \# D_i}
  \right] = -\delta \ind(i < \tau)
\]
where $\delta = [2 - \lambda R^2(3 \sqrt{3} / 4)]/ [1 + \lambda]$.

Hence $\phi(D_{i \wedge \tau}) + (i \wedge \tau) \delta$ is a supermartingale.
As noted above, the result then becomes 
a simple consequence of the OST.
\end{proof}

\section{Experimental Results}
\label{SEC:experimental}

This theoretical result increases the known lower bound for the 
value of $\lambda$ where the clan of descendants is finite, but this
is still just a lower bound.  Computer experiments can estimate 
this critical value of $\lambda$ more precisely.

\begin{figure}[ht]
\begin{center}
\includegraphics[height=4in]{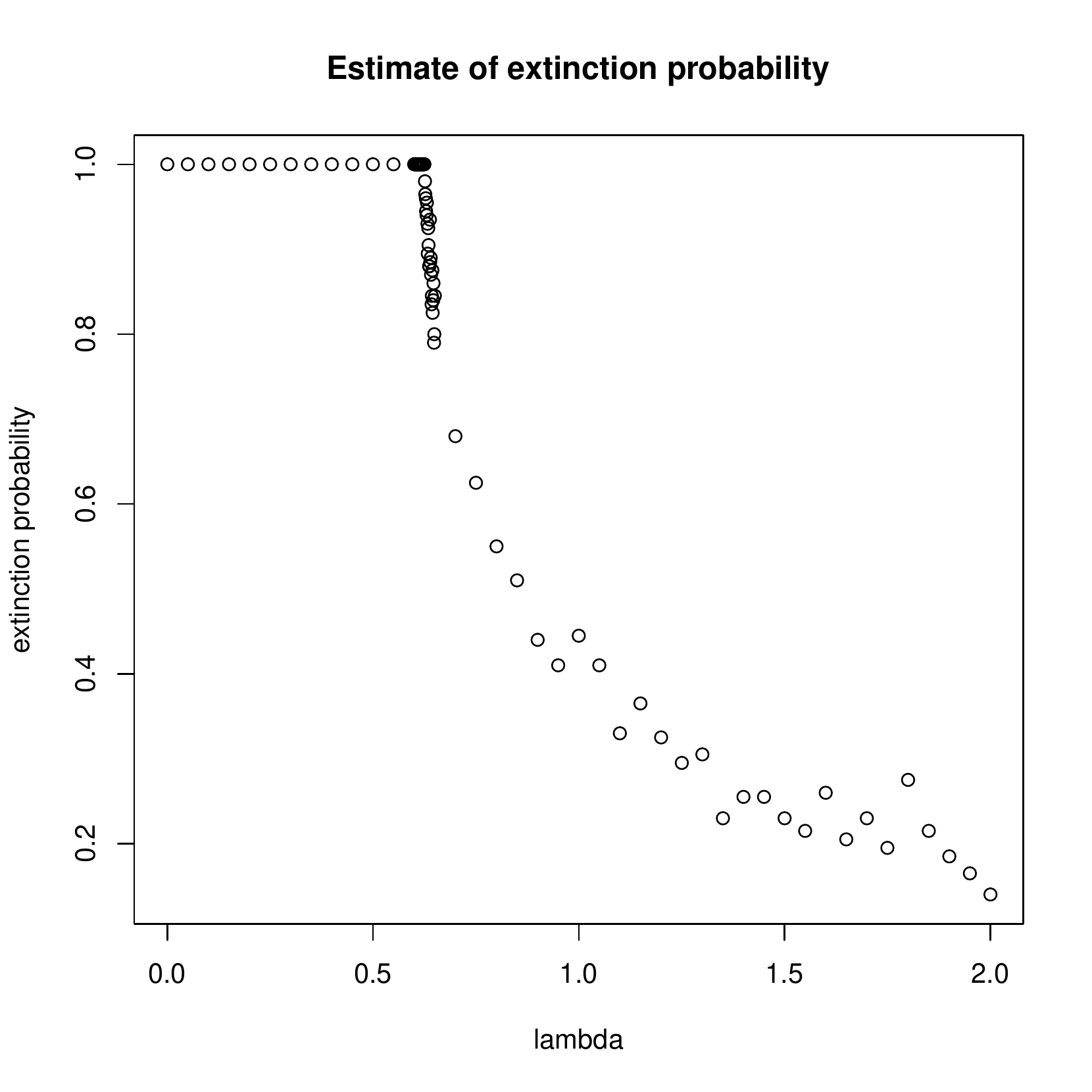}
\end{center}
\caption{Estimates use 200 trials, maximum size of cod 750 points}
\label{FIG:extinction}
\end{figure}

For the estimates in this section, the following protocol was used.  We
began a clan of descendants on the infinite plane from a single point,
and recorded whether the clan died out or reached a size of 750.  This was
repeated 200 times, and used to estimate the probability that the clan 
dies out for a given value of $\lambda$.  The results indicate that 
somewhere in $[0.625,.626]$, the probability begins to drop from 1 down
towards 0 (see Figure~\ref{FIG:extinction} for how the extinction probability
changes with $\lambda$.)  
This indicates that while the new .4503 theoretical result is
an improvement over the old $.3183$ result, there is still work
to be done to reach the true value.  
Increasing the ceiling size from 750 to 1500 did not 
alter the results within experimental error.

\section{Conclusion}
\label{SEC:conclusion}

By including a term for the area covered by the points in the potential
function, a stronger theoretical lower bound on the artificial phase
transition for dominated coupling from the past applied to 
the hard core gas model has been found.  This method appears to be
very general and should apply to a wide variety of repulsive processes.

\bibliographystyle{plain} 
\bibliography{2010_refs}
\end{document}